\documentclass[11pt]{article}

\usepackage[letterpaper]{geometry}
\usepackage{amsmath}  
\usepackage{amsfonts}
\usepackage{amsthm}
\usepackage{amssymb}
\usepackage[capitalize,nameinlink]{cleveref}
\usepackage{color}
\usepackage{enumerate}
\usepackage{graphicx}
\usepackage{url}
\usepackage{float}

\newtheorem{theorem}{Theorem}
\newtheorem{corollary}[theorem]{Corollary}
\newtheorem{lemma}[theorem]{Lemma}

\newtheorem{example}[theorem]{Example}
\newtheorem{proposition}[theorem]{Proposition}

\numberwithin{equation}{section}

\newcommand{\Z}{\mathbb{Z}}
\newcommand{\mcal}{\mathcal}

\newcommand{\bigabs}[1]{\big\lvert#1\big\rvert}
\newcommand{\biggabs}[1]{\bigg\lvert#1\bigg\rvert}
\newcommand{\conj}[1]{\overline{#1}}
\newcommand{\floor}[1]{\lfloor#1\rfloor}

\newcommand{\om}{\omega}

\long\def\symbolfootnote[#1]#2{\begingroup%
\def\thefootnote{\fnsymbol{footnote}}\footnote[#1]{#2}\endgroup}

\allowdisplaybreaks[4]

\begin{document}

\title{Three-phase Golay sequence and array triads}
\author{Aki Ayukawa Avis \and Jonathan Jedwab}
\date{12 October 2019}
\maketitle

\symbolfootnote[0]{
J.~Jedwab is with Department of Mathematics, 
Simon Fraser University, 8888 University Drive, Burnaby BC V5A 1S6, Canada.
\par
A.~Avis is with Department of Nutrition, Universit{\'e} de Montr{\'e}al,
2900 Boulevard {\'E}douard-Montpetit, Montr{\'e}al QB H3T 1J4, Canada.
\par
Both authors were supported by NSERC.
\par
Email: {\tt jed@sfu.ca}, {\tt aki.avis@umontreal.ca}
\par
}

\begin{abstract}
3-phase Golay sequence and array triads are a natural generalisation of 2-phase Golay sequence pairs, yet their study has until now been largely neglected. 
We present exhaustive search results for 3-phase Golay sequence triads for all lengths up to~$24$, showing that the existence pattern is much richer than that for 2-phase Golay sequence pairs.
We give an elementary proof that there is no 3-phase Golay sequence triad whose length is congruent to 4 modulo~6. 
We give two construction methods for 3-phase Golay array triads, and show how to project 3-phase Golay array triads to lower-dimensional Golay triads.
In this way we explain much of the existence pattern found by exhaustive search.
\end{abstract}

\maketitle

\section{Introduction}
\label{sec:intro}
Golay sequence pairs were introduced by Golay in 1951 to solve a problem in multislit spectrometry~\cite{golay-static} and have since been applied in various digital communcations schemes, including coded aperture imaging \cite{ohyama-honda-tsujiuchi}, optical time domain reflectometry~\cite{nazarathy}, power control for multicarrier wireless transmission~\cite{golay-reed-muller}, and medical ultrasound~\cite{nowicki}. 
They are related to Barker sequences \cite{barker-alternatives} and Reed-Muller codes \cite{golay-reed-muller}. 

We consider a \emph{length $s$ sequence} to be an $1$-dimensional matrix $\mcal{A}=(A_i)$ of complex-valued entries, indexed by an integer $i$, for which
\[
\mbox{$A_i = 0$ if either $i < 0$ or $i \ge s$}.
\]
The {\em aperiodic autocorrelation function\/} of $\mcal{A} = (A_i)$ is 
\[
C_{\mcal{A}}(u) = \sum_{i} A_i \conj{A_{i+u}} \quad \mbox{for integers $u$},
\]
where bar represents complex conjugation. 
A 2-phase Golay sequence pair comprises length~$s$ sequences $\mcal{A} = (A_i)$ and $\mcal{B} = (B_i)$ whose elements lie in the alphabet $\{1,-1\}$, satisfying
\begin{equation}
\big( C_{\mcal{A}} + C_{\mcal{B}} \big) (u) = 0 \quad \mbox{for all $u \ne 0$}.
\end{equation}
Golay sequence pairs have been constructed for all lengths $s$ of the form $2^a 10^b 26^c$, where $a,b,c$ are non-negative integers \cite{turyn-hadamard}, but no examples are known for other lengths.
In view of the wide range of potential applications for such sequences, the definition of Golay sequence pairs has been extended in various ways: from the 2-phase alphabet $\{1, -1\}$ to the $H$-phase alphabet (comprising the $H^{\rm th}$ roots of unity) for even~$H$ \cite{craigen-holzmann-kharaghani,golay-reed-muller,fiedler-jedwab-wiebe-6phase,paterson}, and to quadrature amplitude modulation constellations \cite{chong-venkataramani-tarokh,rossing-tarokh}; from sequence pairs to sequence sets of even size greater than~2 \cite{paterson,schmidt-complementary}; and from sequences to arrays of dimension 2 or greater \cite{dymond-phd,fiedler-jedwab-parker-multidim-golay,jedwab-parker-golay-arrays,ohyama-honda-tsujiuchi}. 

Nonetheless, there is a natural generalisation of 2-phase Golay sequence pairs that, with the exception of a brief study by Frank \cite{frank-polyphase} in 1980, appears to have been largely neglected: from the 2-phase alphabet $\{1,-1\}$ to the 3-phase alphabet $\{1,\om,\om^2\}$ where $\om = e^{2\pi i/3}$, and simultaneously from sequence pairs to sequence triads (sets of size~3). The 3-phase alphabet is attractive for applications, because the small number of phases allows the received signal levels to be easily distinguished at the receiver. The related question of the existence of 3-phase Barker arrays was considered in~\cite{three-phase-barker}.

We shall consider the following questions theoretically and computationally, summarising and extending results contained in the 2008 Master's thesis of the first author \cite{avis-masters}:
\begin{enumerate}
\item 
For which lengths $s$ do $3$-phase Golay sequence triads exist?
\item 
How many $3$-phase length~$s$ Golay sequence triads are there?
\item
How are $3$-phase Golay array triads related to $3$-phase Golay sequence triads?
\item
How can $3$-phase Golay sequence and array triads be constructed?
\end{enumerate}

Frank \cite{frank-polyphase} gave an example of a 3-phase Golay sequence triad of length~$3$ and one of length~$5$, and presented a construction for a 3-phase length~$3s$ Golay sequence triad from one of length~$s$. Some of the other constructions described in \cite{frank-polyphase} can be applied to 3-phase Golay sequence triads, but in general they produce sequences over larger alphabets such as 6-phase or 12-phase.

We shall show in \cref{thm:proj} that the projection to lower dimensions of a multi-dimensional 3-phase Golay array triad is also a 3-phase Golay triad, so that a higher-dimensional Golay array triad can be considered a more fundamental object than its lower-dimensional projections; this parallels the central argument of \cite{fiedler-jedwab-parker-multidim-golay} for Golay pairs.
Our main construction method for Golay array triads is \cref{thm:3x}, which modifies Frank's construction \cite{frank-polyphase} for Golay sequence triads to produce an $(r+1)$-dimensional Golay array triad from an $r$-dimensional Golay array triad. 

The rest of this paper is organised as follows.
In \cref{sec:sequences}, we establish structural properties of 3-phase Golay sequence triads, and in \cref{tab:sequence-counts} present counts of 3-phase length~$s$ Golay sequence triads obtained from exhaustive search for each $s \le 24$.
In \cref{sec:nonexistence}, motivated by computational evidence, we give an elementary proof that there is no 3-phase length~$s$ Golay sequence triad when $s \equiv 4 \pmod{6}$.
In \cref{sec:arrays}, we introduce 3-phase Golay array triads and show how they can be projected to 3-phase Golay array triads in one fewer dimension, deducing counts of 3-phase Golay array triads in \cref{tab:array-counts} from those for sequence triads given in \cref{tab:sequence-counts}.
In \cref{sec:constructions}, we present two constructions of 3-phase Golay array triads, and apply them to a small set of seed triads to explain the existence of many of the triads counted in Tables~\ref{tab:sequence-counts} and~\ref{tab:array-counts}.
In \cref{sec:open}, we pose some open questions motivated by our results.
In the Appendix, we explicitly list some of the Golay sequence and array triads whose existence we were not able to explain.

\section{Golay sequence triads}
\label{sec:sequences}
In this section, we establish structural properties of 3-phase Golay sequence triads and present exhaustive search results for each length up to~$24$.

Throughout, let $\om = e^{2\pi i/3}$.
A length~$s$ sequence $(A_i)$ is a \emph{$3$-phase sequence} if all elements of the set $\{A_i \mid 0 \le i < s\}$ take values in $\{1, \om, \om^2\}$.
An unordered set of three 3-phase length~$s$ sequences $\{\mcal{A}, \mcal{B}, \mcal{C}\}$ is a \emph{Golay sequence triad} if
\begin{equation}
\label{eqn:Golayseqtriad}
\big( C_\mcal{A} + C_\mcal{B} + C_\mcal{C} \big) (u) = 0 \quad \mbox{for all $u \ne 0$}.
\end{equation}
It is sufficient to verify \eqref{eqn:Golayseqtriad} for the integers $u$ satisfying $0 < u < s$, because every complex-valued sequence $\mcal{A}$ satisfies $C_{\mcal{A}}(-u) = \conj{C_{\mcal{A}}(u)}$ for all integers~$u$.

A 3-phase length~$s$ sequence $(A_i)$ corresponds to a length~$s$ sequence $(a_i)$ over~$\Z_3$, where 
\[
\mbox{$A_i = \om^{a_i}$ for each $i$ satisfying $0 \le i < s$}.
\]
The aperiodic autocorrelation function of a sequence over~$\Z_3$ is that of the corresponding $3$-phase sequence.

\begin{example}
\label{ex:seqtriad}
The sequences 
\[
\begin{array}{ccccccc}
\mcal{A} = (A_i) = \big [ \hspace{2mm} 1 & \om^2 & 1     & 1     & \om^2 & 1     \hspace{2mm} \big ], \\[1.5ex]
\mcal{B} = (B_i) = \big [ \hspace{2mm} 1 & \om   & \om^2 & \om^2 & \om^2 & \om   \hspace{2mm} \big ], \\[1.5ex]
\mcal{C} = (C_i) = \big [ \hspace{2mm} 1 & \om   & \om   & \om   & 1     & \om^2 \hspace{2mm} \big ],
\end{array}
\]
satisfy
\[
\begin{array}{cccccc}
\big( C_{\mcal{A}}(u) \mid 0 \leq u < 6 \big) = \big [ \hspace{2mm} 6 & -1     & 1     & 3       & -1     & 1     \hspace{2mm} \big ], \\[1.5ex]
\big( C_{\mcal{B}}(u) \mid 0 \leq u < 6 \big) = \big [ \hspace{2mm} 6 & -\om   & \om   & \om-1   & -\om^2 & \om^2 \hspace{2mm} \big ], \\[1.5ex]
\big( C_{\mcal{C}}(u) \mid 0 \leq u < 6 \big) = \big [ \hspace{2mm} 6 & -\om^2 & \om^2 & \om^2-1 & -\om   & \om   \hspace{2mm} \big ],
\end{array}
\]
\vspace{-0.5ex} \\
and therefore comprise a $3$-phase length~$6$ Golay sequence triad.
The corresponding sequences over~$\Z_3$ are
$(a_i) = \left [ \begin{array}{*{6}{c @{\hspace{1.7mm}} }} 0 & 2 & 0 & 0 & 2 & 0 \end{array} \right ]$,
$(b_i) = \left [ \begin{array}{*{6}{c @{\hspace{1.7mm}} }} 0 & 1 & 2 & 2 & 2 & 1 \end{array} \right ]$,
$(c_i) = \left [ \begin{array}{*{6}{c @{\hspace{1.7mm}} }} 0 & 1 & 1 & 1 & 0 & 2 \end{array} \right ]$.
\end{example}

We will use upper-case letters for 3-phase sequences (``multiplicative notation''), and lower-case letters for sequences over~$\Z_3$ (``additive notation'');
the same letter (for example $A$ and~$a$) will indicate corresponding sequences. We will switch between multiplicative and additive notation as convenient.

Suppose that $\{(a_i), (b_i), (c_i)\}$ is a length~$s$ Golay sequence triad over~$\Z_3$. Since the sequences $(x_i)$ and $(x_i+1)$ over~$\Z_3$ have identical aperiodic autocorrelation function, each of the $3^3$ unordered sets
$\big\{(a_i+\alpha), \, (b_i+\beta), \, (c_i+\gamma)\big\}$ 
is also a Golay sequence triad over~$\Z_3$, as $\alpha, \beta, \gamma$ range over $\{0,1,2\}$. We may therefore assume that
\[
(a_0, b_0, c_0) = (0, 0, 0),
\]
in which case the Golay sequence triad is in \emph{normalised form}.
Condition \eqref{eqn:Golayseqtriad} with $u=s-1$ then forces $\{a_{s-1},b_{s-1},c_{s-1}\} = \{0,1,2\}$, which implies there are exactly $3!$ ordered Golay sequence triads corresponding to each (unordered) Golay sequence triad.

The following result follows directly from the definition of aperiodic autocorrelation function and Golay sequence triad. 
\begin{lemma}
\label{lem:equiv-class-sequence-triad}
\mbox{}
\begin{enumerate}[(i)]

\item
{\rm (Linear Offsets).}
Suppose that $\{(a_i), (b_i), (c_i) \}$ is a length $s$ Golay sequence triad over~$\Z_3$. Then
\[
\big \{ (a_i + ei), \, (b_i + ei), \, (c_i + ei) \big \}
\]
is also a length~$s$ Golay sequence triad over~$\Z_3$ for each $e \in \Z_3$.

\item
{\rm (Reversal).}
Suppose that $\{(a_i), (b_i), (c_i) \}$ is a length $s$ Golay sequence triad over~$\Z_3$. Then
\[
\{(a_{s-1-i}), (b_{s-1-i}), (c_{s-1-i}) \}
\]
is also a length~$s$ Golay sequence triad over~$\Z_3$.

\item
{\rm (Reverse Conjugation).}
Let $(x_i)$ be a length $s$ sequence over~$\Z_3$.
Then the length~$s$ sequence
$(2x_{s-1-i})$ over $\Z_3$ has identical aperiodic autocorrelation function to~$(x_i)$.

\end{enumerate}
\end{lemma}

Suppose that $T$ is a length~$s$ Golay sequence triad over~$\Z_3$.
By \cref{lem:equiv-class-sequence-triad}~$(iii)$, when one sequence $(x_i)$ of $T$ is replaced by the length~$s$ sequence $(2x_{s-1-i})$, the resulting set is also a length~$s$ Golay sequence triad over~$\Z_3$. We say that the Golay sequence triads over~$\Z_3$ obtained from $T$ by applying one or more of the operations described in \cref{lem:equiv-class-sequence-triad}, and then taking the normalised form, are all \emph{equivalent} to~$T$. The size of the equivalence class of a normalised Golay sequence triad divides $3 \cdot 2 \cdot 2^3 = 48$. We take the representative of the equivalence class to be its lexicographically first member. For example, there are exactly three equivalence classes of length~5 Golay sequence triads over~$\Z_3$, each of size 24, and their equivalence class representatives are
\begin{align*}
&\big \{
\left [ \begin{array}{*{5}{c @{\hspace{1.7mm}} }} 0 & 0 & 0 & 1 & 0 \end{array} \right ],\, 
\left [ \begin{array}{*{5}{c @{\hspace{1.7mm}} }} 0 & 1 & 2 & 2 & 1 \end{array} \right ],\, 
\left [ \begin{array}{*{5}{c @{\hspace{1.7mm}} }} 0 & 0 & 2 & 1 & 2 \end{array} \right ]
\big \}, \\[1ex]
&\big \{
\left [ \begin{array}{*{5}{c @{\hspace{1.7mm}} }} 0 & 0 & 1 & 0 & 0 \end{array} \right ],\, 
\left [ \begin{array}{*{5}{c @{\hspace{1.7mm}} }} 0 & 0 & 1 & 2 & 1 \end{array} \right ],\, 
\left [ \begin{array}{*{5}{c @{\hspace{1.7mm}} }} 0 & 0 & 2 & 1 & 2 \end{array} \right ]
\big \}, \\[1ex]
&\big \{
\left [ \begin{array}{*{5}{c @{\hspace{1.7mm}} }} 0 & 0 & 1 & 1 & 0 \end{array} \right ],\, 
\left [ \begin{array}{*{5}{c @{\hspace{1.7mm}} }} 0 & 0 & 1 & 2 & 1 \end{array} \right ],\, 
\left [ \begin{array}{*{5}{c @{\hspace{1.7mm}} }} 0 & 0 & 2 & 0 & 2 \end{array} \right ]
\big \}.
\end{align*}

Golay \cite{golay-complementary} showed that the elements of a 2-phase length~$s$ Golay sequence pair $\{(A_i),\, (B_i)\}$ are related by
\[
A_iB_i = -A_{s-1-i}B_{s-1-i} \quad \mbox{for each $i$ satisfying $0 \le i < s$}.
\]
We next derive a counterpart relationship for the elements of a 3-phase Golay sequence triad. 

\begin{proposition}
\label{prop:sum}
Suppose that $\{(A_i), (B_i), (C_i) \}$ is a $3$-phase length~$s$ Golay sequence triad. Then
\[
A_iB_iC_i = A_{s-1-i}B_{s-1-i}C_{s-1-i} \quad \mbox{for each $i$ satisfying $0 \le i < s$}.
\]
\end{proposition}
\begin{proof}
Let $u$ satisfy $0 < u < s$.
We are given that
\[
\sum_{i=0}^{s-1-u} \big ( A_i\conj{A_{i+u}} + B_i\conj{B_{i+u}} + C_i\conj{C_{i+u}} \big ) = 0,
\]
so the multiset of $3(s-u)$ summands above contains each of $1$, $\om$ and $\om^2$ exactly $s-u$ times. The product of these summands is therefore
\[
\prod_{i=0}^{s-1-u} (A_i\conj{A_{i+u}})(B_i\conj{B_{i+u}})(C_i\conj{C_{i+u}}) 
 = 1^{s-u} \om^{s-u} (\om^2)^{s-u} = 1,
\]
which implies that
\[
\prod_{i=0}^{s-1-u} A_i B_i C_i = \prod_{i=u}^{s-1} A_i B_i C_i.
\]
Since this holds for each $u$ satisfying $0 < u < s$, the required result follows by induction on decreasing $u \le s-1$. 
\end{proof}

The property of a length $s$ Golay sequence triad $\{(a_i),(b_i),(c_i)\}$ over $\Z_3$ corresponding to \cref{prop:sum} is
\begin{equation}
\label{eqn:sum}
a_i + b_i + c_i \equiv a_{s-1-i} + b_{s-1-i} + c_{s-1-i} \pmod{3} \quad \mbox{for each $i$ satisfying $0 \le i < s$}.
\end{equation}
We use this property to determine by exhaustive search the equivalence classes of length~$s$ Golay sequence triads $\{(a_i),(b_i),(c_i)\}$ over $\Z_3$ for each $s \le 24$. We begin by fixing the outermost elements of the sequences of the triad: the normalisation condition gives $(a_0,b_0,c_0)=(0,0,0)$, and we may then assume from the autocorrelation condition $u=s-1$ of~\eqref{eqn:Golayseqtriad} that $(a_{s-1},b_{s-1},c_{s-1}) = (0,1,2)$.
For each value of $i$ in $\{1, 2, \dots, \floor{\frac{s-1}{2}}\}$ taken in ascending order, we then recursively select all values of the outermost undetermined sequence elements $\{a_i, b_i, c_i, a_{s-1-i}, b_{s-1-i}, c_{s-1-i}\}$ that are consistent with both property \eqref{eqn:sum} and the autocorrelation condition $u=s-1-i$ of~\eqref{eqn:Golayseqtriad}. We finally collect the resulting Golay sequence triads into equivalence classes and retain only the representative of each equivalence class.
(Since just the representative of each equivalence class is required, the search space can be further reduced: for example, by taking the
leftmost nonzero entry of the subsequence $(a_0, a_1, \dots, a_i)$ to be~1, and discarding search branches where this subsequence occurs lexicographically after the reverse conjugated subsequence $(2a_{s-1}, 2a_{s-2}, \dots, 2a_{s-1-i})$.)
\cref{tab:sequence-counts} shows the resulting counts of equivalence classes of Golay sequence triads, and of normalised Golay sequence triads. 
The existence pattern for 3-phase Golay sequence triads appears to be much richer than that for 2-phase Golay sequence pairs: within the range $1 \le s \le 24$, a 2-phase length~$s$ Golay sequence pair exists only for $s \in \{2,4,8,10,16,20\}$ \cite{turyn-hadamard}, whereas a 3-phase length~$s$ Golay sequence triad exists for all $s \not \in \{4, 10, 16, 22\}$.

We call a sequence belonging to one or more $3$-phase Golay sequence triads (not necessarily in normalised form) a \emph{$3$-phase Golay sequence}.
For example, each of the sequences $\mcal{A}, \mcal{B}, \mcal{C}$ in \cref{ex:seqtriad} is a 3-phase length~6 Golay sequence.
The first author discussed the use of 3-phase Golay sequences in an orthogonal frequency division multiplexing (OFDM) multicarrier transmission scheme, using their favourable peak-to-mean envelope power ratio properties to formulate an alternative exhaustive search algorithm for Golay sequence triads \cite[Chap.~4]{avis-masters} to the one described above. Since the transmission rate of an OFDM scheme using 3-phase Golay sequences increases with the number of Golay sequences of a given length, \cref{tab:sequence-counts} also displays counts of Golay sequences.

\section{Nonexistence result for $s \equiv 4 \pmod{6}$}
\label{sec:nonexistence}
A striking feature of \cref{tab:sequence-counts} is that there are no 3-phase Golay sequence triads of length 4, 10, 16, and~22. In this section, we explain this observation by proving the following theorem.

\begin{theorem}
\label{thm:4mod6}
There is no $3$-phase length~$s$ Golay sequence triad when $s \equiv 4 \pmod{6}$.
\end{theorem}

\begin{table}[H]
\begin{center}
\vspace{2mm}
\begin{tabular}{|r|r|r|r|r|r|r|r|r|}
\hline
Sequence& \multicolumn{6}{|c|}{\# equivalence classes}			& \# normalised	& \# Golay	\\ \cline{2-7}
length 	& size   & size   & size     & size   & size	& Total		& sequence	& sequences	\\
	& 1      & 8      & 16       & 24     & 48 	&		& triads	&		\\ \hline 

2 	& 1      &        &         &         &         & 1		& 1 		& 9		\\
3 	& 1      & 1      &         &         &         & 2		& 9 		& 27		\\ 
4 	&        &        &         &         &         & 0		& 0 		& 0		\\
5 	&        &        &         & 3       &         & 3		& 72 		& 108		\\
6 	&        & 3      & 3       & 4       &         & 10		& 168 		& 288		\\ 
7 	&        &        &         & 8       & 9       & 17		& 624 		& 792		\\
8 	&        &        &         & 1       & 3       & 4		& 168 		& 306		\\
9 	&        & 15     & 33      & 25      & 32      & 105		& 2784 		& 3708		\\ 
10 	&        &        &         &         &         & 0		& 0 		& 0		\\ 
11 	&        &        &         & 14      & 50      & 64		& 2736 		& 4932		\\
12 	&        &        &         &         & 7       & 7		& 336 		& 756		\\ 
13 	&        &        &         & 16      & 48      & 64		& 2688		& 5850		\\ 
14 	&        &        &         & 23      & 68      & 91		& 3816 		& 8334		\\ 
15 	&        & 36     & 306     & 51      & 126     & 519		& 12456 	& 27576		\\ 
16 	&        &        &         &         &         & 0		& 0 		& 0		\\ 
17 	&        &        &         & 10      & 15      & 25		& 960	 	& 2160		\\ 
18 	&        & 84     & 714     & 184     & 503     & 1485		& 40656 	& 89532		\\ 
19 	&        &        &         & 6       & 11      & 17		& 672	 	& 1458		\\ 
20 	&        &        &         &         & 10      & 10		& 480	 	& 1080		\\
21 	&        & 84     & 2766    & 88      & 1007    & 3945		& 95376 	& 202932	\\ 
22 	&        &        &         &         &         & 0		& 0 		& 0		\\ 
23 	&        &        &         & 2       &         & 2		& 48 		& 108		\\ 
24 	&        & 12     & 750     & 16      & 247     & 1025		& 24336 	& 54756		\\ \hline
\end{tabular} 
\end{center}
\caption{Counts of $3$-phase length~$s$ Golay sequence triads, for each $s \le 24$}
\label{tab:sequence-counts}
\end{table}

Let $\mcal{A}=(A_i)$ be a length~$s$ sequence. The \emph{periodic autocorrelation function} of $\mcal{A}$ is
\[
R_\mcal{A}(u) = \sum_{i=0}^{s-1} A_i \conj{A_{(i+u) \bmod s}} \quad \mbox{for $0 \le u < s$},
\]
so that
\begin{equation}
R_\mcal{A}(u) = C_\mcal{A}(u) + \conj{C_\mcal{A}(s-u)} \quad \mbox{for $0 < u < s$}
\label{eqn:sum-aper}
\end{equation}
by the definition of aperiodic autocorrelation. An unordered set of three 3-phase length~$s$ sequences $\{\mcal{A}, \mcal{B}, \mcal{C}\}$ is a \emph{periodic Golay sequence triad} if
\[
\big( R_\mcal{A} + R_\mcal{B} + R_\mcal{C} \big) (u) = 0 \quad \mbox{for all $u$ satisfying $0 < u < s$}.
\]
 From \eqref{eqn:sum-aper}, the sequences of a Golay sequence triad also form a periodic Golay sequence triad of the same length.

Now we cannot prove \cref{thm:4mod6} by attempting to rule out the existence of a 3-phase periodic Golay sequence triad of length congruent to 4 modulo~6: for example, the sequence set
\[
\big \{
\left [
\begin{array}{*{4}{c @{\hspace{1.7mm}} }}
1 & 1 & \om & \om 
\end{array}
\right ],
\left [
\begin{array}{*{4}{c @{\hspace{1.7mm}} }}
1 & 1 & \om & \om
\end{array}
\right ],
\left [
\begin{array}{*{4}{c @{\hspace{1.7mm}} }}
1 & \om & 1 & \om
\end{array}
\right ]
\big \}.
\]
is a periodic length~4 Golay sequence triad. However, we can constrain the elements of a 3-phase periodic Golay sequence triad as follows, which in combination with \cref{prop:sum} will immediately prove \cref{thm:4mod6}.

\begin{proposition}
\label{prop:periodic}
Suppose that $\{(A_i), (B_i), (C_i)\}$ is a $3$-phase length~$2m$ periodic Golay sequence triad, where $m \equiv 2 \pmod{3}$. Then exactly two of the following three equations hold:
\[
\prod_{i=0}^{m-1}A_{2i} = \prod_{i=0}^{m-1}A_{2i+1}, \qquad
\prod_{i=0}^{m-1}B_{2i} = \prod_{i=0}^{m-1}B_{2i+1}, \qquad
\prod_{i=0}^{m-1}C_{2i} = \prod_{i=0}^{m-1}C_{2i+1}.
\]
\end{proposition}

\begin{proof}
Let $\mcal{A} = (A_i)$, $\mcal{B} = (B_i)$, $\mcal{C} = (C_i)$.
We shall examine the value of $S \bmod 9$, where
\begin{equation}
S = \biggabs{\sum_{i=0}^{2m-1}(-1)^i A_i}^2 + \biggabs{\sum_{i=0}^{2m-1}(-1)^i B_i}^2 + \biggabs{\sum_{i=0}^{2m-1}(-1)^i C_i}^2.
\label{eqn:S-defn}
\end{equation}
For a length $2n$ sequence $\mcal{X}=(X_i)$, we have
\begin{align*}
\biggabs{\sum_{i=0}^{2n-1}(-1)^i X_i}^2 
 &= \sum_{i=0}^{2n-1} (-1)^i X_i \sum_{j=0}^{2n-1} (-1)^j \conj{X_j} \\
 &= \sum_{i=0}^{2n-1} (-1)^i X_i \sum_{u=0}^{2n-1} (-1)^{i+u} \conj{X_{(i+u)\bmod 2n}} 
\intertext{using the re-indexing $j = (i+u)\bmod{2n}$, so that}
\biggabs{\sum_{i=0}^{2n-1}(-1)^i X_i}^2 
 &= \sum_{u=0}^{2n-1} (-1)^u R_\mcal{X}(u).
\end{align*}
It follows from \eqref{eqn:S-defn} that
\begin{align}
S &= \sum_{u=0}^{2m-1} (-1)^u \big( R_\mcal{A} + R_\mcal{B} + R_\mcal{C} \big)(u) \nonumber \\
  &= \big( R_\mcal{A} + R_\mcal{B} + R_\mcal{C} \big) (0) \nonumber
\intertext{because $\{\mcal{A}, \mcal{B}, \mcal{C}\}$ is a periodic Golay triad.
Since $\big( R_\mcal{A} + R_\mcal{B} + R_\mcal{C}\big) (0) = 6m$ and $m \equiv 2 \pmod{3}$, this implies that}
S &\equiv 3 \pmod{9} \label{eqn:Smod9}.
\end{align}

Now for a 3-phase length~$2n$ sequence $\mcal{X}=(X_i)$, by taking $Y_i = X_{2i}$ and $Z_i = X_{2i+1}$ in \cref{lem:diff-roots} given below, we find that $\bigabs{\sum_{i=0}^{2n-1}(-1)^i X_i}^2$ is an integer satisfying the congruence
\[
\biggabs{\sum_{i=0}^{2n-1}(-1)^i X_i}^2 \equiv
  \begin{cases}
    0 \pmod{9} & \mbox{if $\prod_{i=0}^{n-1} X_{2i} = \prod_{i=0}^{n-1} X_{2i+1}$}, \\
    3 \pmod{9} & \mbox{otherwise}.
  \end{cases}
\]
The result then follows from \eqref{eqn:S-defn} and~\eqref{eqn:Smod9}.
\end{proof}

\begin{lemma}
\label{lem:diff-roots}
Let $(Y_i)$ and $(Z_i)$ be $3$-phase length $n$ sequences. Then $\bigabs{\sum_{i=0}^{n-1} (Y_i - Z_i)}^2$ is an integer satisfying the congruence
\[
\biggabs{\sum_{i=0}^{n-1} (Y_i - Z_i)}^2 \equiv 
  \begin{cases}
    0 \pmod{9} & \mbox{if $\prod_{i=0}^{n-1} Y_i = \prod_{i=0}^{n-1} Z_i$}, \\
    3 \pmod{9} & \mbox{otherwise}.
  \end{cases}
\]
\end{lemma}
\begin{proof}
For $j = 0, 1, 2$, let the sequence $(Y_i)$ contain $\om^j$ exactly $\alpha_j$ times and let the sequence $(Z_i)$ contain $\om^j$ exactly $\beta_j$ times. Then 
\begin{equation}
\alpha_0 + \alpha_1 + \alpha_2 = \beta_0 + \beta_1 + \beta_2 = n,
\label{eqn:alpha-beta-sum}
\end{equation}
and
\begin{align*}
\sum_{i=0}^{n-1}(Y_i-Z_i) 
 &= (\alpha_0-\beta_0) + \om (\alpha_1-\beta_1) + \om^2 (\alpha_2-\beta_2)\\ 
 &= a + \om b,
\end{align*}
using $\om^2 = -1 -\om$ and writing $a = (\alpha_0-\alpha_2) - (\beta_0-\beta_2)$ and $b = (\alpha_1-\alpha_2) - (\beta_1-\beta_2)$.
Therefore
\begin{align}
\biggabs{\sum_{i=0}^{n-1}(Y_i-Z_i)}^2
 &= a^2+b^2-ab, \nonumber
\intertext{which is an integer because $a$ and $b$ are integers. Since $a+b \equiv 0 \pmod{3}$ by \eqref{eqn:alpha-beta-sum}, this implies}
\biggabs{\sum_{i=0}^{n-1}(Y_i-Z_i)}^2 
  &\equiv 3b^2 \pmod{9}. \label{eqn:S3b2}
\end{align}

Now $\prod_{i=0}^{n-1}Y_i = 1^{\alpha_0} \om^{\alpha_1} (\om^2)^{\alpha_2} = \om^{\alpha_1-\alpha_2}$, and likewise $\prod_{i=0}^{n-1}Z_i = \om^{\beta_1-\beta_2}$, so that 
\[
\prod_{i=0}^{n-1}Y_i = \prod_{i=0}^{n-1}Z_i \quad \mbox{if and only if} \quad b \equiv 0 \pmod{3}
\]
by the definition of $b$. Combination with \eqref{eqn:S3b2} gives the required result.
\end{proof}

\begin{proof}[Proof of \cref{thm:4mod6}]
Suppose, for a contradiction, that $\{(A_i), (B_i), (C_i)\}$ is a 3-phase length $2m$ Golay sequence triad, where $m \equiv 2 \pmod{3}$. Then by \cref{prop:sum},
\begin{equation}
\prod_{i=0}^{m-1}A_{2i}B_{2i}C_{2i} = \prod_{i=0}^{m-1}A_{2i+1}B_{2i+1}C_{2i+1}.
\label{eqn:left-right}
\end{equation}
But $\{(A_i), (B_i), (C_i)\}$ is also a 3-phase length~$2m$ periodic Golay sequence triad by \eqref{eqn:sum-aper}, so we obtain a contradiction to~\eqref{eqn:left-right} from \cref{prop:periodic}.
\end{proof}

\section{Golay array triads}
\label{sec:arrays}
In this section, we extend the definition of a Golay sequence triad to multiple dimensions, and explore the relationship between multi-dimensional Golay array triads and Golay sequence triads. 

We consider an $s_1 \times \dots \times s_r$ \emph{array} to be an $r$-dimensional matrix $\mcal{A}=(A_{i_1, \dots, i_r})$ of complex-valued entries, indexed by integers $i_1, \dots, i_r$, for which
\[
\mbox{$A_{i_1, \dots, i_r} = 0$ if, for at least one $k \in \{1, \dots, r\}$, either $i_k < 0$ or $i_k \ge s_k$}.
\]
The {\em aperiodic autocorrelation function\/} of $\mcal{A}$ is 
\[
C_{\mcal{A}}(u_1, \dots, u_r) = \sum_{i_1,\dots,i_r} A_{i_1, \dots, i_r} \conj{A_{i_1+u_1, \dots, i_r+u_r}} \quad \mbox{for integers $u_1, \dots, u_r$}.
\]
This function satisfies 
\begin{equation}
\label{eqn:conj}
C_{\mcal{A}}(-u_1, \dots, -u_r) = \conj{C_{\mcal{A}}(u_1, \dots, u_r)} \quad \mbox{for all~$(u_1, \dots, u_r)$}.
\end{equation}
The following definitions are each analogous to those for sequences:
a 3-phase array; the $s_1 \times \dots \times s_r$ array $(a_{i_1, \dots,i_r})$ over $\Z_3$ corresponding to a 3-phase $s_1 \times \dots \times s_r$ array $(A_{i_1, \dots, i_r})$; and the aperiodic autocorrelation function of $(a_{i_1, \dots, i_r})$.

An unordered set of three 3-phase $s_1 \times \dots \times s_r$ arrays $\{\mcal{A}, \mcal{B}, \mcal{C}\}$ is a \emph{Golay array triad} if
\[
\big( C_{\mcal{A}} + C_{\mcal{B}} + C_{\mcal{C}} \big) (u_1, \ldots, u_r) = 0 \quad
			\mbox{for all $(u_1, \ldots, u_r) \neq (0, \ldots, 0)$},
\]
and an array $\mcal{A}$ is called a {\em Golay array\/} if it belongs to one or more 3-phase Golay array triads.

\begin{example}
\label{ex:arraytriad}
The arrays 
\[
\mcal{A} = (a_{i,j}) = \left [ \begin{array}{*{3}{c @{\hspace{1.7mm}} }} 0 & 0 & 2 \\ 2 & 0 & 0 \end{array} \right ], \quad
\mcal{B} = (b_{i,j}) = \left [ \begin{array}{*{3}{c @{\hspace{1.7mm}} }} 0 & 2 & 2 \\ 1 & 2 & 1 \end{array} \right ], \quad
\mcal{C} = (c_{i,j}) = \left [ \begin{array}{*{3}{c @{\hspace{1.7mm}} }} 0 & 1 & 0 \\ 1 & 1 & 2 \end{array} \right ] 
\]
over~$\Z_3$ satisfy
\begin{align*}
\big( C_{\mcal{A}}(u,v) \mid 0 \le u < 2, -3 < v < 3 \big) &= 
\begin{bmatrix} -1 & 1 & 6 & 1 & -1 \\ 1 & -1 & 0 & 2 & 1 \end{bmatrix}, \\[1ex]
\big( C_{\mcal{B}}(u,v) \mid 0 \le u < 2, -3 < v < 3 \big) &= 
\begin{bmatrix} -\om & \om^2 & 6 & \om & -\om^2 \\ \om & -\om^2 & 0 & 2\om & \om^2 \end{bmatrix}, \\[1ex]
\big( C_{\mcal{C}}(u,v) \mid 0 \le u < 2, -3 < v < 3 \big) &= 
\begin{bmatrix} -\om^2 & \om & 6 & \om^2 & -\om \\ \om^2 & -\om & 0 & 2\om^2 & \om \end{bmatrix}],
\end{align*}
and therefore comprise a $2 \times 3$ Golay array triad over~$\Z_3$. 
The $2 \times 3$ array $(a_{i,j})$ can instead by represented as the $3 \times 2$ array $(a'_{i,j})$, where $a'_{i,j} = a_{j,i}$ for all $i,j$. However, we do not consider arrays obtained by reordering dimensions to be distinct: they are different formal representations of the same object.
\end{example}

Since the arrays $(x_{i_1,\dots,i_r})$ and $(x_{i_1,\dots,i_r}+1)$ over~$\Z_3$ have identical aperiodic autocorrelation function, an $s_1 \times \dots \times s_r$  Golay sequence triad 
$\{(a_{i_1, \dots, i_r}),\, (b_{i_1, \dots, i_r}),\, (c_{i_1, \dots, i_r})\}$ over~$\Z_3$ may be assumed to satisfy
\[
(a_{0,\dots,0},\, b_{0,\dots,0},\, c_{0,\dots,0}) =(0,0,0),
\]
in which case we say it is in \emph{normalised form}.

The following result is the multi-dimensional version of \cref{lem:equiv-class-sequence-triad}.

\begin{lemma}
\label{lem:equiv-class-array-triad}
\mbox{}
\begin{enumerate}[(i)]

\item
{\rm (Linear Offsets).}
Suppose that $\{(a_{i_1, \dots, i_r}),\, (b_{i_1, \dots, i_r}),\, (c_{i_1, \dots, i_r})\}$ 
is an $s_1 \times \dots \times s_r$ Golay array triad over~$\Z_3$.
Then 
\[
\big \{ 
\big (a_{i_1,\dots,i_r} + e_1 i_1 + \dots + e_r i_r \big),\, 
\big (b_{i_1,\dots,i_r} + e_1 i_1 + \dots + e_r i_r \big),\, 
\big (c_{i_1,\dots,i_r} + e_1 i_1 + \dots + e_r i_r \big)
\big \}
\]
is also an $s_1 \times \dots \times s_r$ Golay array triad over~$\Z_3$
for all $e_1, \dots, e_r \in \Z_3$.

\item
{\rm (Reversal).}
Suppose that $\{(a_{i_1, \dots, i_r}),\, (b_{i_1, \dots, i_r}),\, (c_{i_1, \dots, i_r})\}$ 
is an $s_1 \times \dots \times s_r$ Golay array triad over~$\Z_3$.
Then for each $k \in \{1, \dots, r\}$,
\[
\{(a_{i_1, \dots, s_k-1-i_k, \dots, i_r}),\, (b_{i_1, \dots, s_k-1-i_k, \dots, i_r}),\, (c_{i_1, \dots, s_k-1-i_k, \dots, i_r})\}
\]
is also an $s_1 \times \dots \times s_r$ Golay array triad over~$\Z_3$.

\item
{\rm (Reverse Conjugation).}
Let $(x_{i_1,\dots,i_r})$ be an $s_1 \times \dots \times s_r$ array over~$\Z_3$.
Then the $s_1 \times \dots \times s_r$ array 
$(2x_{s_1-1-i_1,\dots,s_r-1-i_r})$ over $\Z_3$ has identical aperiodic autocorrelation function to~$(x_{i_1,\dots,i_r})$.
\end{enumerate}
\end{lemma}

The Golay array triads over~$\Z_3$ obtained by applying one or more of the operations described in \cref{lem:equiv-class-array-triad}, and then normalising, form an equivalence class whose size divides $3^r \cdot 2^r \cdot 2^3 = 2^{r+3}\cdot 3^r$. We take the representative of the equivalence class to be its lexicographically first member. 
For example, there are exactly three equivalence class of $2 \times 7$ Golay array triads over~$\Z_3$, each of size 288, and their equivalence class representatives are
\begin{align*}
&\bigg \{
\left [ \begin{array}{*{7}{c @{\hspace{1.7mm}} }} 0 & 0 & 0 & 1 & 0 & 2 & 2 \\ 0 & 0 & 2 & 0 & 1 & 1 & 0 \end{array} \right ],\,
\left [ \begin{array}{*{7}{c @{\hspace{1.7mm}} }} 0 & 2 & 0 & 1 & 2 & 0 & 0 \\ 0 & 2 & 2 & 2 & 0 & 2 & 1 \end{array} \right ],\,
\left [ \begin{array}{*{7}{c @{\hspace{1.7mm}} }} 0 & 1 & 0 & 2 & 1 & 1 & 1 \\ 0 & 1 & 2 & 2 & 2 & 0 & 2 \end{array} \right ]
\bigg \}, \\[1ex]
&\bigg \{
\left [ \begin{array}{*{7}{c @{\hspace{1.7mm}} }} 0 & 0 & 0 & 1 & 0 & 2 & 2 \\ 2 & 0 & 2 & 1 & 0 & 0 & 0 \end{array} \right ],\,
\left [ \begin{array}{*{7}{c @{\hspace{1.7mm}} }} 0 & 0 & 2 & 0 & 1 & 1 & 0 \\ 2 & 0 & 1 & 1 & 1 & 2 & 1 \end{array} \right ],\,
\left [ \begin{array}{*{7}{c @{\hspace{1.7mm}} }} 0 & 0 & 1 & 0 & 2 & 0 & 1 \\ 2 & 0 & 0 & 2 & 2 & 1 & 2 \end{array} \right ]
\bigg \}, \\[1ex]
&\bigg \{
\left [ \begin{array}{*{7}{c @{\hspace{1.7mm}} }} 0 & 0 & 0 & 1 & 0 & 2 & 2 \\ 2 & 2 & 0 & 2 & 1 & 2 & 0 \end{array} \right ],\,
\left [ \begin{array}{*{7}{c @{\hspace{1.7mm}} }} 0 & 1 & 1 & 1 & 1 & 0 & 2 \\ 1 & 1 & 2 & 0 & 1 & 2 & 1 \end{array} \right ],\,
\left [ \begin{array}{*{7}{c @{\hspace{1.7mm}} }} 0 & 1 & 0 & 0 & 2 & 2 & 0 \\ 1 & 1 & 1 & 0 & 2 & 1 & 2 \end{array} \right ]
\bigg \}.
\end{align*}

We now introduce an invertible mapping that reduces the number of dimensions of an array by exactly one and that maps Golay array triads to Golay array triads.
Our method is modelled on that of \cite{jedwab-parker-golay-arrays} for 2-phase Golay array pairs, as subsequently used in~\cite{fiedler-jedwab-parker-multidim-golay} and~\cite{gibson-jedwab-even} for $H$-phase Golay array pairs where $H$ is even.
For an $s_1 \times \dots \times s_r$ array $\mcal{A} = (A_{i_1, \dots, i_r})$ (where $r \ge 2$), the \emph{projection} $\psi_{1,2}(\mcal{A})$ is the 
$s_1s_2 \times s_3 \dots \times s_r$ array $(B_{i, i_3, \ldots, i_r})$ given by
\[
B_{i_1 + s_1 i_2, i_3, \dots, i_r} = A_{i_1, \dots, i_r}, \quad
		\mbox{where $0 \le i_1 < s_1$}.
\]
The same definition of $\psi_{1,2}$ holds for an array over~$\Z_3$.
For example, $\psi_{1,2}$ maps the three $2 \times 3$ arrays over~$\Z_3$ of \cref{ex:arraytriad} to the three length~$6$ sequences over~$\Z_3$ of \cref{ex:seqtriad}.
For distinct $k, \ell \in \{1, \dots, r\}$, the array $\psi_{k,\ell}(\mcal{A})$ is defined similarly by replacing the array argument~$i_\ell$ by $i_k + s_k i_\ell$ and removing the array argument~$i_k$.
The mapping $\psi_{k,\ell}$ replaces the $s_k \times s_\ell$ ``slice'' of $\mcal{A}$ formed from dimensions~$k$ and~$\ell$ by the sequence obtained when the elements of the slice are listed column by column.

The following lemma expresses the aperiodic autocorrelation function
$C_{\psi_{1,2}(\mcal{A})}$ as the sum of two terms involving $C_{\cal A}$ (one or both of which might be trivially zero, according to the values of the arguments).
We use this to prove in \cref{thm:proj} that the existence of an $r$-dimensional Golay array triad implies the existence of an $(r-1)$-dimensional Golay array triad.

\begin{lemma}
[{\cite[Lemma~10]{jedwab-parker-golay-arrays}}]
\label{lem:sum}
Let $\mcal{A}$ be a $3$-phase $s_1 \times \cdots \times s_r$ array, where $r \ge 2$.  Then
\begin{align*}
\lefteqn{C_{\psi_{1,2}({\cal A})}(u_1+s_1u_2, u_3, \ldots, u_r)} \hspace{35mm} \\
 & = C_{\cal A}(u_1, \ldots, u_r) + C_{\cal A}(u_1-s_1, u_2+1, u_3, \ldots, u_r) \mbox{ for $0 \le u_1 < s_1$}.
\end{align*}
\end{lemma}

\begin{theorem}[Projection mapping]
\label{thm:proj}
Suppose that $\{\mcal{A}, \mcal{B}, \mcal{C}\}$ is a $3$-phase $s_1 \times \cdots \times s_r$ Golay array triad, where $r \ge 2$. Then 
$\big\{\psi_{1,2}(\mcal{A}), \psi_{1,2}(\mcal{B}), \psi_{1,2}(\mcal{C}) \big\}$ 
is a $3$-phase $s_1 s_2 \times s_3 \times \cdots \times s_r$ Golay array triad.
\end{theorem}

\begin{proof}
Let $u_1, \ldots, u_r$ be integers, where $0 \leq u_1 < s_1$
and $(u_1, \ldots, u_r) \neq (0, \ldots, 0)$.
By Lemma~\ref{lem:sum},
\begin{align*}
\lefteqn{\big( C_{\psi_{1,2}(\mcal{A})} + C_{\psi_{1,2}(\mcal{B})} + C_{\psi_{1,2}(\mcal{C})} \big) (u_1+s_1u_2, u_3, \ldots, u_r)} \hspace{15mm} \\
&= \big( C_{\cal A} + C_{\cal B} + C_{\cal C} \big) (u_1, \ldots, u_r) + 
   \big( C_{\cal A} + C_{\cal B} + C_{\cal C} \big) (u_1-s_1, u_2+1, u_3, \ldots, u_r) \\
&=  0
\end{align*}
because $\{\mcal{A}, \mcal{B}, \mcal{C}\}$ is a 3-phase $s_1 \times \dots \times s_r$ Golay array triad. So
$\big\{\psi_{1,2}(\mcal{A}), \psi_{1,2}(\mcal{B}), \psi_{1,2}(\mcal{C}) \big\}$ 
is a 3-phase 
$s_1s_2 \times s_3 \times \cdots \times s_r$ Golay array triad.
\end{proof}

The effect of the projection mapping~$\psi_{1,2}$ in \cref{thm:proj} is to ``join'' dimension 1 of a Golay array triad to dimension~2. Likewise, the projection mapping~$\psi_{k,\ell}$ for distinct $k, \ell \in \{1,\dots,r\}$ joins dimension $k$ of a Golay array triad to dimension~$\ell$.

The following corollary arises from repeated application of \cref{thm:proj}.
\begin{corollary}
\label{cor:arraytoseq}
Suppose there exists a $3$-phase $s_1 \times \dots \times s_r$ Golay array triad. Then there exists a $3$-phase length $\prod_{k=1}^r s_k$ Golay sequence triad.
\end{corollary}

Combination of \cref{thm:4mod6} and \cref{cor:arraytoseq} gives the following nonexistence result.

\begin{corollary}
\label{cor:4mod6array}
There is no $3$-phase $s_1 \times \dots \times s_r$ Golay array triad when $\prod_{k=1}^r s_k \equiv 4 \pmod{6}$.
\end{corollary}

We use \cref{thm:proj} to determine the equivalence classes of all 3-phase \mbox{$s_1 \times \dots \times s_r$} Golay array triads for which $\prod_{k=1}^r s_k \le 24$. 
For example, to determine the equivalence classes of $2 \times 9$ Golay array triads, let $\psi_{1,2}$ be the projection mapping from $2 \times 9$ arrays to length~$18$ sequences.
Apply the inverse mapping $\psi_{1,2}^{-1}$ to each length~$18$ Golay sequence triad equivalence class representative (as previously determined, and summarised in \cref{tab:sequence-counts}), retaining those $2 \times 9$ triads having the Golay property. Collect the resulting Golay array triads into equivalence classes and retain only the representative of each equivalence class.
To then determine the equivalence classes of $2 \times 3 \times 3$ Golay triads, let $\psi_{2,3}$ be the projection mapping from $2 \times 3 \times 3$ arrays to $2 \times 9$ arrays.
Apply the inverse mapping $\psi_{2,3}^{-1}$ to each $2 \times 9$ Golay array triad equivalence class representative and proceed similarly. (We can alternatively determine the equivalence classes of $2 \times 3 \times 3$ Golay array triads from the $3 \times 6$ Golay array triads, obtained in turn from the length 18 Golay sequence triads.) 
Likewise, we determine from the length~$20$ Golay sequence triads that there are no $2 \times 10$ and no $4 \times 5$ Golay array triads, either of which implies by \cref{thm:proj} that there are also none of size $2 \times 2 \times 5$.

\cref{tab:array-counts} shows the resulting counts of equivalence classes of Golay array triads, normalised Golay array triads, and Golay arrays.
By \cref{cor:4mod6array}, there are no $s_1 \times \dots \times s_r$ Golay array triads for which $\prod_{k=1}^r s_k \in \{4,10,16,22\}$.

\begin{table}[h]
\begin{center}
\vspace{2mm}
\begin{tabular}{|r|r|r|r|r|r|r|r|r|r|r|}
\hline
Array 			& \multicolumn{8}{c|}{\# equivalence classes}			& \# normalised	& \# Golay	\\ \cline{2-9}
size			& size & size & size & size & size & size & size& Total		& array		& arrays	\\ 
			& 24   & 48   & 72   & 96   & 144  & 288  & 576	&      		& triads      	& 		\\ \hline

$2 \times 3$		& 1    & 1    &      &      &      &	  &	& 2		& 72		& 162		\\
$2 \times 4$		&      &      &      &      &      &	  &	& 0		& 0 		& 0  		\\
$3 \times 3$		& 1    & 7    &      & 3    &      &	  &	& 11		& 648  		& 1350		\\ 
$2 \times 6$		&      &      &      &      &      &	  &	& 0		& 0 		& 0  		\\
$3 \times 4$		&      &      &      &      &      &	  &	& 0		& 0 		& 0  		\\
$2 \times 7$		&      &      &      &      &      & 3	  &	& 3   		& 864   	& 1944		\\
$3 \times 5$		&      & 18   &      & 45   &      &	  &	& 63 		& 5184 		& 11664	\\
$3 \times 6$		& 4    & 64   & 4    & 147  & 18   & 8	  &	& 245		& 22464  	& 49788		\\
$2 \times 9$ 		&      & 18   &      & 45   & 18   & 18	  &	& 99  		& 12960   	& 29160		\\
$2 \times 3 \times 3$	&      &      &      &      & 2    & 9    & 4	& 15     	& 5184  	& 11664	\\
$2 \times 10$		&      &      &      &      &      &	  &	& 0		& 0 		& 0  		\\
$4 \times 5$		&      &      &      &      &      &	  &	& 0		& 0 		& 0  		\\
$3 \times 7$		&      & 42   &      & 447  &      & 9	  &	& 498  		& 47520		& 101088	\\ 
$2 \times 12$		&      &      &      &      &      &	  &	& 0		& 0 		& 0  		\\
$3 \times 8$		&      & 6    &      & 123  &      &	  &	& 129 		& 12096 	& 27216		\\ 
$4 \times 6$		&      &      &      &      &      &	  &	& 0		& 0 		& 0  		\\ \hline
\end{tabular} 
\end{center}
\caption{Counts of $3$-phase $s_1 \times \dots \times s_r$ Golay array triads, where $\prod_{k=1}^r s_k \le 24$}
\label{tab:array-counts}
\end{table}

\section{Constructions of Golay array triads}
\label{sec:constructions}
In this section, we present two constructions for Golay array triads, and apply them to a small set of seed Golay sequence and array triads to explain the existence of many of the Golay triads counted in Tables~\ref{tab:sequence-counts} and~\ref{tab:array-counts}.

The \emph{aperiodic cross-correlation function} of $s_1 \times \dots \times s_r$ arrays $\mcal{A} = (A_{i_1,\dots,i_r})$ and $\mcal{B} = (B_{i_1,\dots,i_r})$ is
\[
C_{\mcal{A},\,\mcal{B}}(u_1,\dots,u_r) = \sum_{i_1,\dots,i_r} A_{i_1,\dots,i_r} \conj{B_{i_1+u_1,\dots,i_r+u_r}} \quad \mbox{for integers $u_1,\dots,u_r$}
\]
(so that $C_{\mcal{A},\mcal{A}}(u_1,\dots,u_r) = C_{\mcal{A}}(u_1,\dots,u_r)$ for all~$(u_1,\dots,u_r)$).
Let 
$\mcal{A} = (A_{i_1,\dots,i_r})$, 
$\mcal{B} = (B_{i_1,\dots,i_r})$, 
$\mcal{C} = (C_{i_1,\dots,i_r})$ be $s_1 \times \ldots \times s_r$ arrays. Write $\begin{bmatrix} \mcal{A} \\ \mcal{B} \\ \mcal{C} \end{bmatrix}$ for the $3 \times s_1 \times \dots \times s_r$ array $\mcal{D} = (D_{i,i_1,\dots,i_r})$ defined by
\[
D_{i,i_1,\dots,i_r} = 
\begin{cases}
A_{i_1,\dots,i_r} & \mbox{for $i=0$} \\
B_{i_1,\dots,i_r} & \mbox{for $i=1$} \\
C_{i_1,\dots,i_r} & \mbox{for $i=2$}.
\end{cases}
\]
It follows directly from the definitions that
\begin{equation}
\label{eqn:DABC} 
\left .
\begin{array}{rl}
C_{\mcal{D}}(0,u_1,\dots,u_r) &= \big( C_{\mcal{A}} + C_{\mcal{B}} + C_{\mcal{C}} \big) (u_1,\dots,u_r), \\[1ex]
C_{\mcal{D}}(1,u_1,\dots,u_r) &= \big( C_{\mcal{A},\,\mcal{B}} + C_{\mcal{B},\,\mcal{C}} \big) (u_1,\dots,u_r), \\[1ex]
C_{\mcal{D}}(2,u_1,\dots,u_r) &= C_{\mcal{A},\,\mcal{C}}(u_1,\dots,u_r). 
\end{array}
\right \}
\end{equation}

We now present our main construction method for Golay array triads.

\begin{theorem}[Increase dimension]
Suppose that 
$\{\mcal{A}, \mcal{B}, \mcal{C}\}$ 
is a $3$-phase $s_1 \times \cdots \times s_r$ Golay array triad and let
$\mcal{U}= \begin{bmatrix} \mcal{A} \\ \mcal{B}       \\ \mcal{C}       \end{bmatrix}$,
$\mcal{V}= \begin{bmatrix} \mcal{A} \\ \om \mcal{B}   \\ \om^2 \mcal{C} \end{bmatrix}$,
$\mcal{W}= \begin{bmatrix} \mcal{A} \\ \om^2 \mcal{B} \\ \om \mcal{C}   \end{bmatrix}$.
Then $\{\mcal{U}, \mcal{V}, \mcal{W}\}$ is a $3$-phase $3 \times s_1 \times \dots \times s_r$ Golay array triad.
\label{thm:3x}
\end{theorem}

\begin{proof}
For all integers $u_1, \dots, u_r$, by \eqref{eqn:DABC} we have
\begin{align*}
\lefteqn{\big( C_{\mcal{U}} + C_{\mcal{V}} + C_{\mcal{W}} \big) (1,u_1,\dots,u_r)} \hspace{35mm} \\
&=            \big( C_{\mcal{A},\,\mcal{B}} + C_{\mcal{B},\,\mcal{C}} \big) (u_1,\dots,u_r) + 
              \big( C_{\mcal{A},\,\om \mcal{B}} + C_{\om \mcal{B},\,\om^2 \mcal{C}} \big) (u_1,\dots,u_r) + \\
&\phantom{==} \big( C_{\mcal{A},\,\om^2 \mcal{B}} + C_{\om^2 \mcal{B},\,\om \mcal{C}} \big) (u_1,\dots,u_r) \\
&=            (1+\om^2+\om)C_{\mcal{A},\,\mcal{B}}(u_1,\dots,u_r) + (1+\om^2+\om)C_{\mcal{B},\,\mcal{C}}(u_1,\dots,u_r) \\
&= 0
\end{align*}
and
\begin{align*}
\big( C_{\mcal{U}} + C_{\mcal{V}} + C_{\mcal{W}} \big) (2,u_1,\dots,u_r)
&= \big( C_{\mcal{A},\,\mcal{C}} + C_{\mcal{A},\,\om^2 \mcal{C}} + C_{\mcal{A},\,\om \mcal{C}} \big) (u_1,\dots,u_r) \\
&= (1+\om+\om^2)C_{\mcal{A},\,\mcal{C}}(u_1,\dots,u_r) \\
&= 0.
\end{align*}
Combining these results, we see that the condition
\[
\big( C_{\mcal{U}} + C_{\mcal{V}} + C_{\mcal{W}} \big) (u,u_1,\dots,u_r) = 0 \quad \mbox{for all integers $u, u_1, \dots, u_r$}
\]
holds for $u \in \{1,2\}$, and therefore by \eqref{eqn:conj} it also holds for $u \in \{-1,-2\}$. It remains to consider the case $u=0$.

For all integers $u_1, \dots, u_r$ for which $(u_1, \dots, u_r) \ne (0,\dots,0)$, by \eqref{eqn:DABC} we have
\begin{align*}
\lefteqn{\big( C_{\mcal{U}} + C_{\mcal{V}} + C_{\mcal{W}} \big) (0,u_1,\dots,u_r)} \hspace{30mm} \\
&=            \big( C_{\mcal{A}} + C_{\mcal{B}} + C_{\mcal{C}} \big) (u_1,\dots,u_r) +
              \big( C_{\mcal{A}} + C_{\om \mcal{B}} + C_{\om^2 \mcal{C}} \big) (u_1,\dots,u_r) + \\
&\phantom{==} \big( C_{\mcal{A}} + C_{\om^2 \mcal{B}} + C_{\om \mcal{C}} \big) (u_1,\dots,u_r) \\
&= 3\big( C_{\mcal{A}} + C_{\mcal{B}} + C_{\mcal{C}} \big) (u_1,\dots,u_r) \\
&= 0
\end{align*}
because $\{\mcal{A},\mcal{B},\mcal{C}\}$ is a Golay array triad.
Therefore $\{U, V, W\}$ is a $3$-phase $3 \times s_1 \times \dots \times s_r$ Golay array triad.
\end{proof}

\cref{thm:3x} is a simplification and reinterpretation of a construction given by Frank \cite[p.~644]{frank-polyphase} for producing a length~$3s$ Golay sequence triad over $\Z_3$ from a length~$s$ Golay sequence triad over~$\Z_3$.
We can readily recover Frank's result by applying \cref{thm:3x} to a length~$s$ Golay sequence triad, followed by projection of the resulting $3 \times s$ Golay array triad to a length $3s$ Golay sequence triad using \cref{thm:proj}. This example illustrates our contention that a higher-dimensional Golay array triad can be considered a more fundamental object than its lower-dimensional projections. Indeed, we now explain many of the entries of Tables~\ref{tab:sequence-counts} and~\ref{tab:array-counts} by systematically applying \cref{thm:3x} (which introduces exactly one new dimension) in conjunction with \cref{thm:proj} (which removes exactly one dimension), starting from a small set of seed Golay sequence and array triads.

Before applying \cref{thm:proj} to a set of equivalence class representatives, we replace each representative by its full equivalence class. We then apply the projection mappings $\psi_{k,\ell}$ and $\psi_{\ell,k}$ for all distinct $\{k,\ell\}$ to each element of the class, retaining only the representative of each resulting equivalence class of Golay sequence or array triads. 
Before applying \cref{thm:3x} to a set of equivalence class representatives, we likewise replace each representative by its full equivalence class, but then also remove the assumption of normalised form and take all $3!$ orderings of the triad sequences (thereby multiplying each class size by a factor of $3^3 \cdot 3!$). We then apply the construction of \cref{thm:3x} to each element of the expanded class, retaining only the representative of each resulting equivalence class of Golay array triads.

\cref{fig:constructions} displays the result of applying these two constructions in conjunction. The seed Golay sequence triads are the single equivalence class of length 1 and of length 2 (both of which are trivial), all equivalence classes of length 5, 7 and 8, and one of the 10 equivalence classes of length~6. The seed Golay array triads are nine of the 99 equivalence classes of size $2 \times 9$.
\cref{fig:constructions} also shows the result of applying \cref{thm:cross-cor} below, which uses length~$s$ Golay sequence triads having special cross-correlation properties to construct $2 \times s$ and $3 \times s$ Golay array triads.

\begin{figure}[H]
\vspace{-5mm}
\hspace{-4mm}
\includegraphics[scale=0.74]{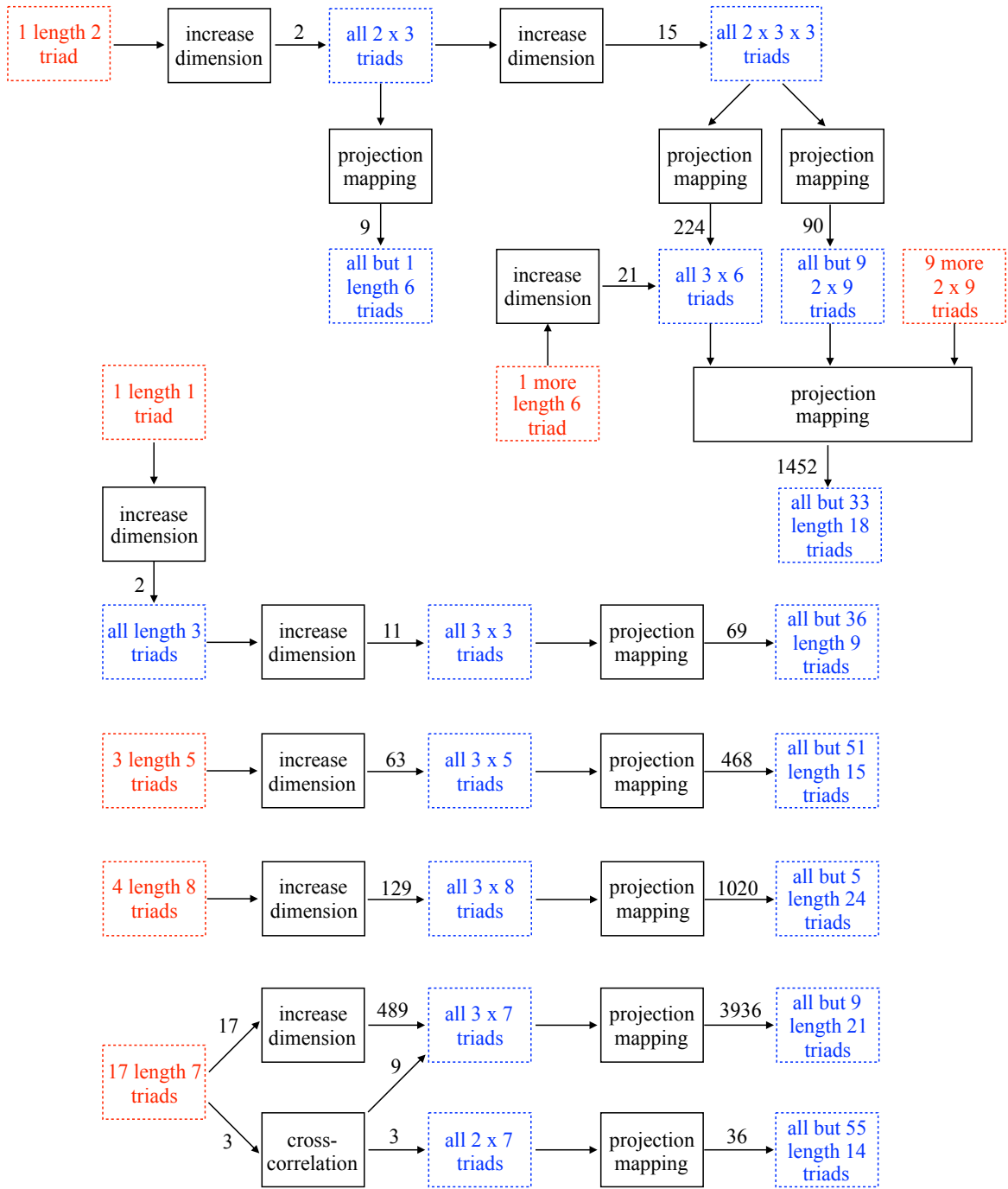}
\vspace{-27mm}
  \caption{Equivalence classes of 3-phase Golay triads whose existence is explained by the constructions of Theorems
\ref{thm:3x} (increase dimension), 
\ref{thm:proj} (projection mapping), and
\ref{thm:cross-cor} (cross-correlation).
Input seed triads are in red, constructed triads are in blue.}
  \label{fig:constructions}
\end{figure}

\begin{theorem}[Cross-correlation]
\label{thm:cross-cor}
Suppose that both
$\{\mcal{A}_1$, $\mcal{B}_1$, $\mcal{C}_1\}$ and
$\{\mcal{A}_3$, $\mcal{B}_3$, $\mcal{C}_3\}$
are $3$-phase length $s$ Golay sequence triads and that, for all integers~$u$, 
\[
\big( C_{\mcal{A}_1,\,\mcal{A}_3} + C_{\mcal{B}_1,\,\mcal{B}_3} + C_{\mcal{C}_1,\,\mcal{C}_3} \big) (u) = 0.
\]
Then 
$\left \{ 
\begin{bmatrix} \mcal{A}_1 \\ \mcal{A}_3 \end{bmatrix},
\begin{bmatrix} \mcal{B}_1 \\ \mcal{B}_3 \end{bmatrix},
\begin{bmatrix} \mcal{C}_1 \\ \mcal{C}_3 \end{bmatrix}
\right \}$
is a $3$-phase $2 \times s$ Golay array triad.

Suppose that $\{\mcal{A}_2$, $\mcal{B}_2$, $\mcal{C}_2\}$ 
is also a $3$-phase length $s$ Golay sequence triad and that, for all integers~$u$, 
\[
\big( C_{\mcal{A}_1,\,\mcal{A}_2} + C_{\mcal{A}_2,\,\mcal{A}_3} +
C_{\mcal{B}_1,\,\mcal{B}_2} + C_{\mcal{B}_2,\,\mcal{B}_3} +
C_{\mcal{C}_1,\,\mcal{C}_2} + C_{\mcal{C}_2,\,\mcal{C}_3} \big) (u) = 0.
\]
Then
$\Bigg \{
\begin{bmatrix} \mcal{A}_1 \\ \mcal{A}_2 \\ \mcal{A}_3 \end{bmatrix},
\begin{bmatrix} \mcal{B}_1 \\ \mcal{B}_2 \\ \mcal{B}_3 \end{bmatrix},
\begin{bmatrix} \mcal{C}_1 \\ \mcal{C}_2 \\ \mcal{C}_3 \end{bmatrix}
\Bigg \}$
is a $3$-phase \mbox{$3 \times s$} Golay array triad.
\end{theorem}

\begin{proof}
We give the proof for the $3 \times s$ array triad; the proof for the $2 \times s$ array triad is similar.
Let 
$
\mcal{U} = \begin{bmatrix} \mcal{A}_1 \\ \mcal{A}_2 \\ \mcal{A}_3 \end{bmatrix},
\mcal{V} = \begin{bmatrix} \mcal{B}_1 \\ \mcal{B}_2 \\ \mcal{B}_3 \end{bmatrix},
\mcal{W} = \begin{bmatrix} \mcal{C}_1 \\ \mcal{C}_2 \\ \mcal{C}_3 \end{bmatrix}.
$
For all integers~$u \ne 0$, by \eqref{eqn:DABC} we have
\begin{align*}
\lefteqn{\big( C_{\mcal{U}} + C_{\mcal{V}} + C_{\mcal{W}} \big) (0,u)} \hspace{15mm} \\
 &= \big( C_{\mcal{A}_1} + C_{\mcal{A}_2} + C_{\mcal{A}_3} \big) (u) + 
    \big( C_{\mcal{B}_1} + C_{\mcal{B}_2} + C_{\mcal{B}_3} \big) (u) + 
    \big( C_{\mcal{C}_1} + C_{\mcal{C}_2} + C_{\mcal{C}_3} \big) (u) = 0
\end{align*}
because each of
$\{\mcal{A}_1$, $\mcal{B}_1$, $\mcal{C}_1\}$,
$\{\mcal{A}_2$, $\mcal{B}_2$, $\mcal{C}_2\}$,
$\{\mcal{A}_3$, $\mcal{B}_3$, $\mcal{C}_3\}$ 
is a Golay sequence triad.
Furthermore, for all integers $u$ we have
\begin{align*}
\lefteqn{\big( C_{\mcal{U}} + C_{\mcal{V}} + C_{\mcal{W}} \big) (1,u)} \hspace{20mm} \\
 &= \big( C_{\mcal{A}_1,\,\mcal{A}_2} + C_{\mcal{A}_2,\,\mcal{A}_3} \big) (u) + 
    \big( C_{\mcal{B}_1,\,\mcal{B}_2} + C_{\mcal{B}_2,\,\mcal{B}_3} \big) (u) +
    \big( C_{\mcal{C}_1,\,\mcal{C}_2} + C_{\mcal{C}_2,\,\mcal{C}_3} \big) (u) = 0
\end{align*}
by assumption, and
\[
\big( C_{\mcal{U}} + C_{\mcal{V}} + C_{\mcal{W}} \big) (2,u)
 = C_{\mcal{A}_1,\,\mcal{A}_3}(u) + C_{\mcal{B}_1,\,\mcal{B}_3}(u) + C_{\mcal{C}_1,\,\mcal{C}_3}(u) = 0
\]
by assumption. It follows from \eqref{eqn:conj} that 
$\{\mcal{U}, \mcal{V}, \mcal{W}\}$ is a $3$-phase \mbox{$3 \times s$} Golay array triad.
\end{proof}

To construct $3 \times 7$ and $2 \times 7$ Golay array triads using \cref{thm:cross-cor}, note that the three (inequivalent) 3-phase length~7 Golay sequence triads 
\begin{align*}
\mcal{T}_1 &= \{\mcal{U}_1, \mcal{V}_1, \mcal{W}_1 \} =
\left \{ 
\left [ \begin{array}{*{7}{c @{\hspace{1.6mm}} }} 1 & 1 & \om   & 1     & \om^2 & \om^2 & 1     \end{array} \right ],
\left [ \begin{array}{*{7}{c @{\hspace{1.6mm}} }} 1 & 1 & 1     & \om^2 & 1     & \om   & \om   \end{array} \right ],
\left [ \begin{array}{*{7}{c @{\hspace{1.6mm}} }} 1 & 1 & \om^2 & 1     & \om   & 1     & \om^2 \end{array} \right ] 
\right \},\\[1ex]
\mcal{T}_2 &= \{\mcal{U}_2, \mcal{V}_2, \mcal{W}_2 \} = 
\left \{ 
\left [ \begin{array}{*{7}{c @{\hspace{1.6mm}} }} 1 & \om & \om   & \om   & 1     & \om   & \om^2 \end{array} \right ],
\left [ \begin{array}{*{7}{c @{\hspace{1.6mm}} }} 1 & \om & 1     & \om^2 & \om   & 1     & 1     \end{array} \right ],
\left [ \begin{array}{*{7}{c @{\hspace{1.6mm}} }} 1 & \om & \om^2 & \om^2 & \om^2 & \om^2 & \om   \end{array} \right ] 
\right \},\\[1ex]
\mcal{T}_3 &= \{\mcal{U}_3, \mcal{V}_3, \mcal{W}_3 \} =
\left \{ 
\left [ \begin{array}{*{7}{c @{\hspace{1.6mm}} }} 1 & \om^2 & \om   & \om & \om   & 1     & \om   \end{array} \right ],
\left [ \begin{array}{*{7}{c @{\hspace{1.6mm}} }} 1 & \om^2 & 1     & \om & \om^2 & \om^2 & \om^2 \end{array} \right ],
\left [ \begin{array}{*{7}{c @{\hspace{1.6mm}} }} 1 & \om^2 & \om^2 & 1   & 1     & \om   & 1     \end{array} \right ] 
\right \}
\end{align*}
satisfy, for all integers~$u$,
\begin{align*}
\big( C_{\mcal{U}_1,\,\mcal{U}_2} + C_{\mcal{V}_1,\,\mcal{V}_2} + C_{\mcal{W}_1,\,\mcal{W}_2} \big) (u) & = 0, \\
\big( C_{\mcal{U}_2,\,\mcal{U}_3} + C_{\mcal{V}_2,\,\mcal{V}_3} + C_{\mcal{W}_2,\,\mcal{W}_3} \big) (u) & = 0, \\
\big( C_{\mcal{U}_3,\,\mcal{U}_1} + C_{\mcal{V}_3,\,\mcal{V}_1} + C_{\mcal{W}_3,\,\mcal{W}_1} \big) (u) & = 0.
\end{align*}
We may therefore apply \cref{thm:cross-cor} to construct three 3-phase $2 \times 7$ Golay array triads by taking 
$\big( \{\mcal{A}_1,\mcal{B}_1,\mcal{C}_1\}, \{\mcal{A}_3,\mcal{B}_3,\mcal{C}_3\} \big)$
to be 
$(\mcal{T}_1,\mcal{T}_2)$ or 
$(\mcal{T}_2,\mcal{T}_3)$ or 
$(\mcal{T}_3,\mcal{T}_1)$, and these three constructed triads are inequivalent.
Using the identity 
\[
C_{\mcal{B},\mcal{A}}(u) = \overline{C_{\mcal{A},\mcal{B}}(-u)} \quad \mbox{for all $u$}
\]
for complex-valued sequences $\mcal{A}, \mcal{B}$ of equal length, 
we may also apply \cref{thm:cross-cor} to construct nine 3-phase $3 \times 7$ Golay array triads by taking
$\big( \{\mcal{A}_1,\mcal{B}_1,\mcal{C}_1\}, \{\mcal{A}_2,\mcal{B}_2,\mcal{C}_2\}, \{\mcal{A}_3,\mcal{B}_3,\mcal{C}_3\} \big)$
to be 
$(\om^e \mcal{T}_1, \mcal{T}_2, \mcal{T}_3 )$ or
$(\mcal{T}_2, \mcal{T}_3, \om^e \mcal{T}_1 )$ or
$(\mcal{T}_3, \om^e \mcal{T}_1, \mcal{T}_2 )$
for each $e \in \Z_3$,
and these nine constructed triads are inequivalent.

We see from \cref{fig:constructions} that, starting from a small set of seed Golay sequence and array triads, we can use Theorems~\ref{thm:proj}, \ref{thm:3x}, and \ref{thm:cross-cor} to explain the existence of all $2 \times 3$,\, $3 \times 3$,\, $2 \times 7$,\, $3 \times 5$,\, $3 \times 6$,\, $2 \times 3 \times 3$, $3 \times 7$, and $3 \times 8$ Golay array triads,
and all but nine of the 99 equivalence classes of $2 \times 9$ Golay array triads. We can also explain a large proportion of the equivalence classes of length 3, 6, 9, 14, 15, 18, 21, and 24 Golay sequence triads, including all but nine of the 3945 equivalence classes of length~21 and all but five of the 1025 equivalence classes of length~24. 
The counts of equivalence classes recorded in Tables~\ref{tab:sequence-counts} and~\ref{tab:array-counts} that are not explained by these constructions are summarised in \cref{tab:unexplained}. 
The representative of some of these unexplained equivalence classes is listed in the Appendix.

\section{Open Questions}
\label{sec:open}
In this section, we pose some open questions motivated by our results.
The counts shown in \cref{tab:sequence-counts}, and the nonexistence result of \cref{thm:4mod6}, suggest a natural question:

\begin{enumerate}
\item[Q1.]
Does there exist a $3$-phase length~$s$ Golay sequence triad if and only if $s \not \equiv 4 \pmod{6}$?
\end{enumerate}
The counts shown in \cref{tab:array-counts} suggest a possible generalisation of \cref{thm:4mod6} to each of the dimensions of a Golay array triad:
\begin{enumerate}
\item[Q2.]
Does the existence of a 3-phase $s_1 \times \dots \times s_r$ Golay array triad imply that $s_k \not \equiv 4 \pmod{6}$ for each $k$?
\end{enumerate}
The repetition of some of the counts shown in \cref{tab:array-counts} suggests a possible connection:
\begin{enumerate}
\item[Q3.]
Is there a convincing explanation for the total count of normalised array triads and of Golay arrays being identical for 3-phase Golay triads of sizes $3 \times 5$ and $2 \times 3 \times 3$?
\end{enumerate}
The unexplained equivalence classes counted in \cref{tab:unexplained} prompt the question:
\begin{enumerate}
\item[Q4.]
Can new constructions be found to account for the equivalence classes of Golay sequence and array triads marked as ``some'' or ``none'' in \cref{tab:unexplained}?
\end{enumerate}

\begin{table}[H]
\begin{center}
\vspace{2mm}
\begin{tabular}{|r|r|r|r|r|r|r|c|}
\hline
Sequence 		& Total	\#	& \multicolumn{5}{c|}{\# unexplained equivalence classes}   
										& none/some/all	\\ \cline{3-7}
or array     		& equivalence 	& size & size & size & size & Total  	& explained, or	\\ 
size			& classes	& 1    & 24   & 48   & 288  &       	& seeds		\\ \hline

$2$			& 1		& 1    &      &      &      & 1 	& trivial seed 	\\ 
$5$			& 3		&      &  3   &      &      & 3 	& seeds 	\\
$7$			& 17		&      &  8   & 9    &      & 17 	& seeds 	\\ 
$8$			& 4		&      & 1    & 3    &      & 4 	& seeds 	\\ \hline

$11$			& 64		&      & 14   & 50   &      & 64 	& none 		\\
$12$			& 7		&      &      & 7    &      & 7 	& none 		\\
$13$			& 64		&      & 16   & 48   &      & 64 	& none 		\\
$17$			& 25		&      & 10   & 15   &      & 25 	& none 		\\
$19$			& 17		&      & 6    & 11   &      & 17 	& none 		\\
$20$			& 10		&      &      & 10   &      & 10 	& none 		\\ 
$23$			& 2		&      & 2    &      &      & 2 	& none 		\\ \hline

$3$			& 2		&      &      &      &      & 0 	& all 		\\
$6$			& 10		&      & 1    &      &      & 1 	& some (*)	\\
$9$			& 105		&      & 10   & 26   &      & 36 	& some 		\\
$14$			& 91		&      & 23   & 32   &      & 55 	& some		\\
$15$			& 519		&      & 15   & 36   &      & 51 	& some		\\
$18$			& 1485		&      & 28   & 5    &      & 33 	& some		\\
$21$			& 3945		&      & 4    & 5    &      & 9 	& some		\\ 
$24$			& 1025		&      & 4    & 1    &      & 5 	& some		\\ \hline

$2 \times 3$		& 2		&      &      &      &      & 0 	& all 		\\
$3 \times 3$		& 11		&      &      &      &      & 0 	& all 		\\
$2 \times 7$		& 3		&      &      &      &      & 0     	& all 		\\
$3 \times 5$		& 63		&      &      &      &      & 0		& all 		\\
$3 \times 6$		& 245		&      &      &      &      & 0      	& all 		\\
$2 \times 9$ 		& 99		&      &      &      & 9    & 9		& some (*)	\\
$2 \times 3 \times 3$	& 15		&      &      &      &      & 0		& all 		\\
$3 \times 7$		& 498		&      &      &      &      & 0   	& all 		\\ 
$3 \times 8$		& 129		&      &      &      &      & 0   	& all 		\\ \hline
\end{tabular} 
\end{center}
\caption{Counts of 3-phase Golay sequence and array triads whose existence is not explained by the constructions of Theorems~\ref{thm:proj}, \ref{thm:3x}, \ref{thm:cross-cor}. (*) indicates that the unexplained equivalence class(es) are used as seeds for other sizes.}
\label{tab:unexplained}
\end{table}

\section*{Appendix: Unexplained Golay sequence and array triads}
In this Appendix, we give the representative of those equivalence classes of Golay triads that are not explained by the constructions of \cref{sec:constructions}, for lengths 6, 21, and 24, and for size $2 \times 9$.

The single unexplained equivalence class of length~6 Golay sequence triads over~$\Z_3$ has representative
\[
\big \{
\left [
\begin{array}{*{6}{c @{\hspace{1.7mm}} }}
0 & 0 & 0 & 1 & 1 & 0 
\end{array}
\right ],
\left [
\begin{array}{*{6}{c @{\hspace{1.7mm}} }}
0 & 2 & 0 & 2 & 2 & 1 
\end{array}
\right ],
\left [
\begin{array}{*{6}{c @{\hspace{1.7mm}} }}
0 & 1 & 2 & 2 & 0 & 2 
\end{array}
\right ]
\big \}.
\]

The nine unexplained equivalence classes of length~21 Golay sequence triads over~$\Z_3$ have representatives
\begin{align*}
& \big \{
\left [ 0 0 0 0 1 2 0 1 2 2 2 1 0 2 1 2 1 2 2 1 0 \right ],\,
\left [ 0 1 1 2 0 0 1 2 2 1 2 1 1 1 2 1 2 0 2 2 1 \right ],\,
\left [ 0 0 0 0 2 1 1 0 1 0 2 1 1 0 2 0 0 0 0 1 2 \right ]
\big \}, \\[0.5ex]
& \big \{
\left [0 0 0 0 1 2 2 1 0 1 2 0 0 1 0 1 0 0 2 1 0 \right ],\,
\left [0 0 0 0 1 2 2 1 0 0 0 1 2 0 1 2 1 2 1 2 1 \right ],\, 
\left [0 0 0 0 1 2 2 1 0 2 2 2 1 2 2 0 2 1 0 0 2 \right ] 
\big \}, \\[0.5ex]
& \big \{
\left [0 0 0 1 2 2 2 1 2 2 1 0 0 0 0 1 1 0 0 1 0 \right ],\, 
\left [0 0 0 1 2 1 0 0 2 0 1 0 1 1 2 0 2 1 2 2 1 \right ],\, 
\left [0 0 0 1 2 0 0 2 0 1 1 0 0 2 0 2 0 2 1 0 2 \right ] 
\big \}, \\[0.5ex]
& \big \{
\left [0 0 1 1 0 1 0 2 1 2 1 2 1 2 0 1 0 1 1 0 0 \right ],\, 
\left [0 0 1 1 0 0 0 2 0 0 1 1 2 1 1 0 1 0 2 2 1 \right ],\, 
\left [0 0 2 2 0 0 0 1 0 0 2 2 1 2 2 0 2 0 1 1 2 \right ] 
\big \}, \\[0.5ex]
& \big \{
\left [0 0 1 1 0 2 2 0 2 0 2 0 2 1 2 2 0 1 1 0 0 \right ],\, 
\left [0 0 1 1 2 2 2 1 0 0 2 0 2 0 1 2 1 1 2 1 1 \right ],\, 
\left [0 0 1 0 0 1 0 0 1 0 1 0 2 0 1 1 1 0 0 2 2 \right ] 
\big \}, \\[0.5ex]
& \big \{
\left [0 0 1 1 2 0 0 1 2 2 1 0 2 0 2 1 0 0 2 1 0 \right ],\, 
\left [0 0 0 0 2 2 2 2 0 0 0 0 2 0 1 0 1 2 1 2 1 \right ],\, 
\left [0 0 2 1 2 2 1 2 0 1 0 0 1 2 0 0 2 0 0 0 2 \right ] 
\big \}, \\[0.5ex]
& \big \{
\left [0 1 1 1 1 0 0 1 2 2 2 0 1 2 1 1 2 2 2 1 0 \right ],\, 
\left [0 0 0 2 1 2 1 1 2 0 1 0 2 0 1 0 0 2 1 0 1 \right ],\, 
\left [0 0 0 1 2 1 0 2 2 0 2 2 0 2 2 2 2 0 1 0 2 \right ] 
\big \}, \\[0.5ex]
& \big \{
\left [0 1 1 1 2 1 1 2 1 2 0 1 0 0 2 2 2 1 2 1 0 \right ],\, 
\left [0 0 0 2 2 0 1 1 2 2 2 2 1 0 2 2 1 0 2 0 1 \right ],\, 
\left [0 0 1 1 0 2 0 0 0 0 0 1 2 0 1 2 1 0 1 0 2 \right ] 
\big \}, \\[0.5ex]
& \big \{
\left [0 1 1 1 2 2 0 2 2 1 2 2 1 1 1 0 2 1 2 1 0 \right ],\, 
\left [0 0 0 1 1 1 1 0 2 2 1 2 2 1 1 2 1 2 1 0 1 \right ],\, 
\left [0 0 1 1 0 0 1 2 0 0 1 2 1 2 0 1 0 0 2 0 2 \right ]
\big \}.
\end{align*}

The five unexplained equivalence classes of length~24 Golay sequence triads over~$\Z_3$ have representatives
\begin{align*}
& \big \{
\left [ 0 0 0 0 1 0 1 0 0 1 1 2 0 1 2 0 0 2 1 1 2 1 0 0 \right ],\,
\left [ 0 0 1 2 2 2 0 1 0 1 0 1 1 1 1 1 0 2 1 0 0 2 2 1 \right ],\,
\left [ 0 0 1 1 0 0 2 2 0 1 1 0 2 0 0 2 0 2 0 2 1 2 1 2 \right ]
\big \}, \\[0.5ex]
& \big \{
\left [ 0 0 0 0 1 1 1 1 0 1 0 1 2 0 0 2 2 0 1 2 2 2 1 0 \right ],\,
\left [ 0 0 0 2 1 2 0 1 0 0 2 0 1 2 2 0 0 2 1 0 1 0 2 1 \right ],\,
\left [ 0 0 0 1 1 2 1 1 0 2 1 2 0 1 1 1 1 0 0 1 0 1 0 2 \right ]
\big \}, \\[0.5ex]
& \big \{
\left [ 0 0 0 1 0 0 0 1 1 0 1 2 2 1 2 1 1 0 2 0 1 2 1 0 \right ],\,
\left [ 0 0 0 1 2 0 0 0 2 2 0 0 2 0 1 1 0 1 0 2 2 0 2 1 \right ],\,
\left [ 0 0 0 2 2 2 2 0 1 1 2 1 2 2 0 2 0 1 0 2 1 1 0 2 \right ]
\big \}, \\[0.5ex]
& \big \{
\left [ 0 1 2 0 2 2 2 0 2 0 0 0 2 1 2 2 2 1 0 1 1 2 1 0 \right ],\,
\left [ 0 0 0 1 2 1 1 0 1 2 2 1 0 2 0 0 0 2 0 2 2 2 0 1 \right ],\,
\left [ 0 0 0 0 2 0 1 1 0 0 1 2 1 0 0 1 2 1 0 0 1 1 0 2 \right ]
\big \}, \\[0.5ex]
& \big \{
\left [ 0 1 2 1 2 2 0 1 0 2 1 0 0 1 2 0 1 0 2 2 1 2 1 0 \right ],\,
\left [ 0 0 0 2 0 0 0 2 2 1 1 2 0 0 2 1 2 2 2 2 0 2 0 1 \right ],\,
\left [ 0 0 0 1 0 0 0 1 1 2 2 1 0 0 1 2 1 1 1 1 0 1 0 2 \right ]
\big \}.
\end{align*}

The nine unexplained equivalence classes of $2 \times 9$ Golay array triads over~$\Z_3$ have representatives
\begin{align*}
& \bigg \{
\left [
\begin{array}{*{9}{c @{\hspace{1.7mm}} }}
0 & 0 & 0 & 0 & 0 & 0 & 0 & 1 & 2 \\
2 & 1 & 0 & 2 & 0 & 1 & 2 & 1 & 0
\end{array}
\right ],
\left [
\begin{array}{*{9}{c @{\hspace{1.7mm}} }}
0 & 2 & 0 & 0 & 1 & 1 & 2 & 1 & 2 \\ 
1 & 2 & 2 & 1 & 1 & 0 & 2 & 2 & 1
\end{array}
\right ],
\left [
\begin{array}{*{9}{c @{\hspace{1.7mm}} }}
0 & 1 & 1 & 2 & 0 & 0 & 1 & 0 & 1 \\
2 & 2 & 1 & 1 & 0 & 1 & 0 & 0 & 2 
\end{array}
\right ]
\bigg \}, \\[0.5ex]
& \bigg \{
\left [
\begin{array}{*{9}{c @{\hspace{1.7mm}} }}
0 & 0 & 0 & 0 & 0 & 1 & 2 & 2 & 2 \\
2 & 1 & 0 & 2 & 0 & 2 & 1 & 2 & 0
\end{array}
\right ],
\left [
\begin{array}{*{9}{c @{\hspace{1.7mm}} }}
0 & 1 & 1 & 2 & 1 & 1 & 2 & 1 & 2 \\
1 & 1 & 0 & 0 & 1 & 0 & 2 & 2 & 1
\end{array}
\right ],
\left [
\begin{array}{*{9}{c @{\hspace{1.7mm}} }}
0 & 1 & 0 & 0 & 2 & 1 & 1 & 0 & 0 \\
1 & 1 & 2 & 1 & 2 & 0 & 1 & 1 & 2
\end{array}
\right ]
\bigg \}, \\[0.5ex]
& \bigg \{
\left [
\begin{array}{*{9}{c @{\hspace{1.7mm}} }}
0 & 0 & 0 & 0 & 1 & 0 & 0 & 0 & 2 \\
2 & 0 & 1 & 2 & 1 & 1 & 0 & 2 & 0
\end{array}
\right ],
\left [
\begin{array}{*{9}{c @{\hspace{1.7mm}} }}
0 & 0 & 2 & 0 & 2 & 1 & 1 & 1 & 2 \\
1 & 2 & 2 & 1 & 2 & 0 & 2 & 1 & 1
\end{array}
\right ],
\left [
\begin{array}{*{9}{c @{\hspace{1.7mm}} }}
0 & 0 & 1 & 2 & 0 & 0 & 1 & 1 & 1 \\
2 & 0 & 2 & 1 & 0 & 1 & 1 & 0 & 2
\end{array}
\right ]
\bigg \}, \\[0.5ex]
& \bigg \{
\left [
\begin{array}{*{9}{c @{\hspace{1.7mm}} }}
0 & 0 & 0 & 0 & 1 & 0 & 2 & 2 & 2 \\
0 & 1 & 2 & 0 & 0 & 1 & 2 & 1 & 0 
\end{array}
\right ],
\left [
\begin{array}{*{9}{c @{\hspace{1.7mm}} }}
0 & 2 & 0 & 0 & 1 & 0 & 0 & 1 & 1 \\
1 & 1 & 0 & 1 & 0 & 0 & 2 & 2 & 1 
\end{array}
\right ],
\left [
\begin{array}{*{9}{c @{\hspace{1.7mm}} }}
0 & 1 & 1 & 2 & 1 & 0 & 0 & 1 & 1 \\
0 & 2 & 0 & 2 & 0 & 1 & 0 & 0 & 2 
\end{array}
\right ]
\bigg \}, \\[0.5ex]
& \bigg \{
\left [
\begin{array}{*{9}{c @{\hspace{1.7mm}} }}
0 & 0 & 0 & 0 & 1 & 1 & 1 & 1 & 0 \\
1 & 2 & 0 & 1 & 0 & 1 & 0 & 2 & 0 
\end{array}
\right ],
\left [
\begin{array}{*{9}{c @{\hspace{1.7mm}} }}
0 & 0 & 2 & 0 & 0 & 2 & 2 & 2 & 0 \\
0 & 1 & 1 & 0 & 2 & 0 & 2 & 1 & 1 
\end{array}
\right ],
\left [
\begin{array}{*{9}{c @{\hspace{1.7mm}} }}
0 & 0 & 1 & 2 & 0 & 1 & 2 & 2 & 2 \\
1 & 2 & 1 & 0 & 2 & 1 & 1 & 0 & 2 
\end{array}
\right ]
\bigg \}, \\[0.5ex]
& \bigg \{
\left [
\begin{array}{*{9}{c @{\hspace{1.7mm}} }}
0 & 0 & 0 & 0 & 1 & 2 & 0 & 1 & 1 \\
1 & 0 & 2 & 1 & 1 & 1 & 0 & 2 & 0 
\end{array}
\right ],
\left [
\begin{array}{*{9}{c @{\hspace{1.7mm}} }}
0 & 0 & 2 & 0 & 2 & 2 & 0 & 1 & 0 \\
2 & 1 & 2 & 2 & 2 & 0 & 2 & 1 & 1 
\end{array}
\right ],
\left [
\begin{array}{*{9}{c @{\hspace{1.7mm}} }}
0 & 0 & 1 & 2 & 0 & 2 & 1 & 2 & 0 \\
1 & 0 & 0 & 0 & 0 & 1 & 1 & 0 & 2 
\end{array}
\right ]
\bigg \}, \\[0.5ex]
& \bigg \{
\left [
\begin{array}{*{9}{c @{\hspace{1.7mm}} }}
0 & 0 & 0 & 1 & 0 & 0 & 0 & 0 & 2 \\
0 & 2 & 1 & 1 & 2 & 1 & 2 & 0 & 0 
\end{array}
\right ],
\left [
\begin{array}{*{9}{c @{\hspace{1.7mm}} }}
0 & 0 & 2 & 2 & 1 & 0 & 2 & 2 & 0 \\
0 & 2 & 0 & 2 & 0 & 1 & 1 & 2 & 1 
\end{array}
\right ],
\left [
\begin{array}{*{9}{c @{\hspace{1.7mm}} }}
0 & 0 & 1 & 0 & 2 & 0 & 1 & 1 & 1 \\
0 & 2 & 2 & 0 & 1 & 1 & 0 & 1 & 2 
\end{array}
\right ]
\bigg \}, \\[0.5ex]
& \bigg \{
\left [
\begin{array}{*{9}{c @{\hspace{1.7mm}} }}
0 & 0 & 0 & 1 & 0 & 1 & 0 & 2 & 0 \\
1 & 2 & 0 & 2 & 2 & 1 & 2 & 0 & 0 
\end{array}
\right ],
\left [
\begin{array}{*{9}{c @{\hspace{1.7mm}} }}
0 & 0 & 2 & 2 & 1 & 1 & 2 & 1 & 1 \\
1 & 2 & 2 & 0 & 0 & 1 & 1 & 2 & 1 
\end{array}
\right ],
\left [
\begin{array}{*{9}{c @{\hspace{1.7mm}} }}
0 & 0 & 1 & 0 & 2 & 1 & 1 & 0 & 2 \\
1 & 2 & 1 & 1 & 1 & 1 & 0 & 1 & 2 
\end{array}
\right ]
\bigg \}, \\[0.5ex]
& \bigg \{
\left [
\begin{array}{*{9}{c @{\hspace{1.7mm}} }}
0 & 0 & 0 & 1 & 0 & 2 & 1 & 0 & 1 \\
0 & 1 & 2 & 1 & 1 & 1 & 2 & 0 & 0 
\end{array}
\right ],
\left [
\begin{array}{*{9}{c @{\hspace{1.7mm}} }}
0 & 0 & 2 & 2 & 1 & 2 & 0 & 2 & 2 \\
0 & 1 & 1 & 2 & 2 & 1 & 1 & 2 & 1 
\end{array}
\right ],
\left [
\begin{array}{*{9}{c @{\hspace{1.7mm}} }}
0 & 0 & 1 & 0 & 2 & 2 & 2 & 1 & 0 \\
0 & 1 & 0 & 0 & 0 & 1 & 0 & 1 & 2
\end{array}
\right ]
\bigg \}.
\end{align*}

\bibliographystyle{plain}

\end{document}